\documentclass[12pt,letterpaper,showkeys]{article} 
\usepackage{graphicx}
\usepackage{amssymb}
\usepackage{amsmath}
\usepackage{mathrsfs}
\graphicspath{{figures/}}
\newtheorem{theorem}{Theorem}[section] 
\newtheorem{lemma}[theorem]{Lemma} 
 
\newtheorem{proposition}[theorem]{Proposition} 
\newtheorem{remark}[theorem]{Remark}
\newtheorem{assumption}{Assumption} 

\newcommand{\beq}{\begin{equation}} 
\newcommand{\eeq}{\end{equation}} 
\newcommand{\beqa}{\begin{eqnarray}} 
\newcommand{\eeqa}{\end{eqnarray}} 
\newcommand{\beqas}{\begin{eqnarray*}} 
\newcommand{\eeqas}{\end{eqnarray*}} 
\newcommand{\ba}{\begin{array}} 
\newcommand{\ea}{\end{array}} 
\newcommand{\bi}{\begin{itemize}} 
\newcommand{\ei}{\end{itemize}} 
\newcommand{\gap}{\hspace*{2em}} 
 
\newcommand{\nn}{\nonumber}

\def\eqnok#1{(\ref{#1})}

\def\vgap{\vspace*{.1in}}
\def\QED{\ifhmode\unskip\nobreak\fi\ifmmode\ifinner\else\hskip5pt\fi\fi
  \hbox{\hskip5pt\vrule width5pt height5pt depth1.5pt\hskip1pt}} 
\def\Arg{{\rm Arg}}
\def\bC{{\bar {\cal C}}}
\def\bcG{{\bar {\cal G}}}
\def\bd{{\bar d}}
\def\bh{{\bar h}}

\def\bx{{\bar x}}
\def\by{{\bar y}}

\def\cA{{\cal A}}
\def\cB{{\cal B}}
\def\cC{{\cal C}}
\def\cF{{\cal F}} 
\def\cG{{\cal G}}
\def\cK{{\cal K}}
\def\cL{{\cal L}}
\def\cN{{\cal N}}

\def\cT{{\cal T}}
\def\cX{{\cal X}} 

\def\cone{{\rm cone}}
\def\dist{{\rm dist}}

\def\eps{{\epsilon}}

\def\tx{{\tilde x}}

\title{Sequential Convex Programming Methods for A Class of Structured Nonlinear Programming}  

\author{
	Zhaosong Lu%
\thanks{
Department of Industrial and Systems Engineering, University of Minnesota, USA (email:  {\tt zhaosong@umn.edu}).}
	}

\date{October 10, 2012 (Revised: June 17, 2022)}

\begin{document}

\maketitle

\begin{abstract}


In this paper we study a broad class of structured nonlinear programming (SNLP) problems. In 
particular, we first establish the first-order optimality conditions for them. Then we propose 
sequential convex programming (SCP) methods for solving them in which each iteration 
is obtained by solving a convex programming problem. Under some suitable 
assumptions, we establish that any accumulation point of the sequence generated by the methods is a 
KKT point of the SNLP problems.  In addition, we propose a variant of the SCP method for 
SNLP in which nonmonotone scheme and ``local''  Lipschitz constants of  the associated 
functions are used. A similar convergence result as mentioned above is established.  

\vskip14pt

\vskip14pt

\end{abstract}

\section{Introduction} \label{intro}

In this paper we consider a class of structured nonlinear programming problems in the form of
\beq \label{general-dc}
\ba{ll}
\min & f(x) + p(x) -u(x) \\
\mbox{s.t.} & g_i(x) + q_i(x) -v_i(x) \le 0, \ \ \ i=1, \ldots, m, \\
& x \in \cX,
\ea
\eeq
where  $\cX \subseteq \Re^n$ is a nonempty closed convex set, $f$, $g_i$'s are differentiable in 
$\cX$, and $p$, $u$,  $q_i$'s, $v_i$'s are convex (but not necessarily smooth) in $\cX$. 

Throughout this paper we make the following assumption.
\begin{assumption}
The gradients of $f$ and $g_i$'s are Lipschitz continuous in $\cX$ with constants $L_f \ge 0$ 
and $L_{g_i} \ge0$ for $i=1, \ldots, m$, that is, 
\[
\ba{lcl}
\|\nabla f(x) - \nabla f(y)\| &\le& L_f \|x-y\|, \ \forall x, y \in\cX,  \\ [5pt]
\|\nabla g_i(x) - \nabla g_i(y)\| &\le& L_{g_i} \|x-y\|, \ \forall x, y \in\cX, \ i=1,\ldots, m.
\ea
\]
\end{assumption}

Some special cases of problem \eqref{general-dc} have received considerable attention in 
the literature (see, for example,  \cite{TaAn98,BeTe09,Nes07,TsYu09,WrNoFi09,AuShTe10,HoYaZh11,LuZh12}).
In particular, Nesterov \cite{Nes07} and Beck and Teboulle \cite{BeTe09} considered a special case of 
\eqref{general-dc} with $m=0$, $u \equiv 0$ and $f$ being smooth convex with Lipschitz 
continuous gradient, and they proposed accelerated gradient methods for solving it. Tseng and Yun \cite{TsYu09}, Wright 
et al.\ \cite{WrNoFi09}, and Lu and Zhang \cite{LuZh12} proposed efficient first-order methods for the similar problems 
as studied in \cite{BeTe09,Nes07} with $f$ being smooth but not necessarily convex.  Recently, Auslender et al.\ \cite{AuShTe10} 
studied another special case of \eqref{general-dc}, where $\cX = \Re^n$, $p\equiv 0$, $u\equiv 0$, $q_i\equiv 0$, $v_i 
\equiv 0$ for all $i$, and $f$ and $g_i$'s are smooth with Lipschitz continuous gradient. They 
proposed a gradient-based method so called the moving balls approximation (MBA) method for 
solving the problem.  Very recently, Hong et al.\ \cite{HoYaZh11} studied a sequential convex programming (SCP) approach 
for solving a special case of \eqref{general-dc} with $m=1$, $f\equiv 0$, $g_1 \equiv 0$, and $p$, $u$, $q_1$, $u_1$ being 
smooth convex functions in $\cX$. In addition, a broad subclass of \eqref{general-dc} with $m=0$, $f \equiv 0$, known as DC 
(difference of convex functions) programming, was extensively studied and efficient first-order method was proposed for it (see, for example, \cite{TaAn98,HoTh99}).
 
Recently, a class of nonlinear programming models were widely used for finding a sparse approximate 
solution to a system or a function. They can also be viewed as special cases of \eqref{general-dc}.  In 
particular, they are in the form of
\beq \label{sparse-NLP}
\min\limits_{x \in \Omega} l(x) + \sum^n_{i=1} h(|x_i|),
\eeq
where $l$ is a loss function, $\Omega \subseteq \Re^n$ is a nonempty closed convex set, and $h: \Re_+ \to \Re_+$ is a 
sparsity-induced penalty function. Some popular $h$'s used in the literature 
are listed as follows:
\bi
\item[(i)]  ($l_1$ penalty \cite{Ti96,ChDoSa98,CanTao05}):  $h(t) = \lambda t$ \ $\forall t \ge 0$;
\item[(ii)]  (SCAD penalty \cite{FaLi01}): $h(t) = \left\{
\ba{ll}
\lambda t &  \mbox{if} \  0 \le t \le \lambda, \\ [4pt]
\frac{-t^2+2a\lambda t -\lambda^2}{2(a-1)} &  \mbox{if} \ \lambda < t \le a \lambda, \\ [4pt]
\frac{(a+1)\lambda^2}{2} & \mbox{if } \ t > a \lambda;
\ea\right.
$
\item[(iii)] ($l_q$ penalty \cite{FrFr93,Fu98}): $h(t)=\lambda (t+\eps)^q$ \ $\forall t \ge 0$;
\item[(iv)] (Log penalty \cite{WeElScTi03}): $h(t) = \lambda \log(t+\eps)-\lambda \log(\epsilon)$ \ $\forall t \ge 0$;
\item[(v)] (Capped-$l_1$ penalty \cite{Zhang09}): $h(t) = \left\{
\ba{ll}
\lambda t & \mbox{if} \ 0 \le t < \eta, \\
\lambda \eta & \mbox{if} \ t \ge \eta,
\ea\right.$
\ei 
 where $\lambda>0$, $0<q<1$, $a >1$, $\eta>0$ and $\eps>0$ are parameters. 
One can observe that the above $h$'s are monotonically increasing functions in $[0,\infty)$. 
Moreover, $\lambda t - h(t)$ is convex in $[0,\infty)$ (see \cite{GaRaCa09}). It implies that 
$u(y) = \sum^n_{i=1} (\lambda y_i - h(y_i))$ is convex in $\Re^n_+$. Using the monotonicity of $h$, 
we can see that 
\eqref{sparse-NLP} can be equivalently reformulated as 
\[
\min\{l(x) + \sum^n_{i=1} h(y_i): y \ge |x|, x \in \Omega\}.
\]
Further, by using the definition of $u$, we observe that \eqref{sparse-NLP} is equivalent to 
\[
\min\{l(x) + \lambda \|y\|_1 - u(y): y \ge |x|, x \in \Omega\},
\]
which clearly is a special case of \eqref{general-dc} with $\cX = \{(x,y): y \ge |x|, x \in \Omega\}$.
 
In this paper we provide a comprehensive study on problem \eqref{general-dc}. In particular, we 
first establish the first-order optimality conditions for \eqref{general-dc}. Then we propose 
SCP methods for solving \eqref{general-dc} in which each iteration is obtained by solving a convex programming problem. Under some suitable assumptions, we establish 
that any accumulation point of the sequence generated by the methods is a KKT point of 
\eqref{general-dc}.  In addition, we propose a variant of the SCP method for \eqref{general-dc} in 
which nonmonotone scheme and ``local''  Lipschitz constants of  the associated functions are used. 
A similar convergence result as mentioned above is established.  

The outline of this paper is as follows. In  Subsection \ref{notation} we introduce some notations 
that are used in the paper. In Section \ref{opt-conds} we establish the first-order optimality 
conditions for problem \eqref{general-dc}. In Section \ref{exact} we propose an SCP method  
and its variant for solving \eqref{general-dc} and establish their convergence. 

\subsection{Notation} \label{notation}

Given a nonempty closed convex $\Omega \subseteq \Re^n$, $\cone(\Omega)$ denotes the cone 
generated by $\Omega$. Given an arbitrary point $x \in \Omega$, $\cN_\Omega(x)$ and 
$\cT_\Omega(x)$ denote the normal and tangent cones of $\Omega$ at $x$, respectively. In addition, 
$\dist(y,\Omega)$ denotes the distance between $y \in \Re^n$ and $\Omega$. For a function $h: 
\Omega \to \Re$, $d\in\Re^n$ and $x\in\Omega$, 
$h'(x;d)$ is the directional derivative of $h$ at $x$ along $d$. For a convex function $h$, $\partial 
h(x)$ denotes the subdifferential of $h$ at $x$. Finally, given any $t\in \Re$, we denote its nonnegative part by $t^+$, that is, $t^+=\max(t,0)$.   

\section{First-order optimality conditions} 
\label{opt-conds}
  
In this section we establish the first-order optimality conditions for problem 
\eqref{general-dc}. Given any $x\in \cX$, the set of indices corresponding to the active 
constraints of \eqref{general-dc} at $x$ is denoted by $\cA(x)$, 
that is,
\[
\cA(x) = \{1\le i \le m: \ g_i(x) + q_i(x) -v_i(x) = 0\}.
\]
 
\begin{theorem}
Suppose that $x^*$ is a local minimizer of problem \eqref{general-dc}. Assume that the cone
\beq \label{closedness}
\sum\limits_{i \in \cA(x^*)} \cone(\nabla g_i(x^*) + 
\partial q_i(x^*) - \partial v_i(x^*)) + \cN_{\cX}(x^*) 
\eeq
is closed, and moreover, 
there exists $\bd \in \cT_{\cX}(x^*)$ such that 
\beq \label{slater}
g'_i(x^*;\bd) + q'_i(x^*;\bd) - \inf\limits_{s \in \partial v_i(x^*)} s^T\bd \ < \ 0, \quad 
\forall i \in \cA_\cT(x^*),  
\eeq
where 
\beq \label{AT}
\cA_\cT(x^*) = \{i \in \cA(x^*): g'_i(x^*; d) + q'_i(x^*;d) - v'_i(x^*;d) = 0 \ \mbox{for some} 
\ 0 \neq d \in \cT_{\cX}(x^*)\}.
\eeq
Then, there exists $\lambda^* \in \Re^m$ together with $x^*$ satisfying 
the KKT conditions 
\[
\ba{c}
0 \in \nabla f(x^*) + \partial p(x^*) - \partial u(x^*) + \sum^{m}\limits_{i=1} 
\lambda^*_i[\nabla g_i(x^*) + \partial q_i(x^*) - \partial v_i(x^*)] + \cN_{\cX}(x^*), \\
\lambda^*_i \ge 0, \quad \lambda^*_i [g_i(x^*) + q_i(x^*) - v_i(x^*)] = 0, \quad i=1, \ldots, m.
\ea
\]
\end{theorem} 

\begin{proof} 
For convenience, let 
\[
\ba{lcl}
A &=& -\nabla f(x^*) - \partial p(x^*) + \partial u(x^*), \\ [5pt]
B &=& \sum\limits_{i \in \cA(x^*)} \cone(\nabla g_i(x^*) + 
\partial q_i(x^*) - \partial v_i(x^*)) + \cN_{\cX}(x^*).
\ea
\]
In view of the assumption, one can observe that $A$ and $B$ are closed convex sets.  
We first show that $A \cap B \neq \emptyset$. Suppose for contradiction that $A \cap B = \emptyset$. 
It then follows from the well-known separation theorem that there exists $0 \neq d \in \Re^n$ 
such that 
\beq \label{sep-ineq}
\inf\limits_{s\in A} d^T s \ge 1, \quad\quad \sup\limits_{s\in B} d^T s \le 0.
\eeq 
By the definition of $A$ and the first inequality of \eqref{sep-ineq}, one has
\beq \label{obj-dir}
\ba{lcl}
f'(x^*; d) + p'(x^*;d) - u'(x^*;d)  &=&  d^T\nabla f(x^*)  
+ \sup\limits_{s\in \partial p(x^*)} d^T s - \sup\limits_{s\in \partial u(x^*)} d^T 
s \\ [8pt]
&\le& \sup \limits_{s\in A} (-d^T s) \ \le \ -1 \ < \ 0.   
\ea
\eeq
In addition, it follows from the definition of $B$ and the second inequality of \eqref{sep-ineq} that 
$d \in (\cN_{\cX}(x^*))^\circ = \cT_{\cX}(x^*)$ and 
 \[
 \sup\{d^T s: s \in \nabla g_i(x^*) + 
\partial q_i(x^*) - \partial v_i(x^*)\} \ \le \ 0, \quad \forall i \in \cA(x^*),
\] 
which implies that  
\[
g'_i(x^*; d) + q'_i(x^*;d) - v'_i(x^*;d) \ \le \ 0, \ \forall i \in \cA(x^*).
\] 
Since $d \in \cT_{\cX}(x^*)$, there exist a positive sequence $\{t_k\} \downarrow 0$ and 
a sequence $\{x^k\} \subseteq \cX$ such that $x^k = x^* + t_k d + o(t_k)$. We next consider 
two cases to derive a contradiction.

Case 1): Suppose that $g'_i(x^*; d) + q'_i(x^*;d) - v'_i(x^*;d)  <  0$ for all $i \in \cA(x^*)$. 
It then follows that for every $i\in\cA(x^*)$,
\[
\ba{lcl}
g_i(x^k) + q_i(x^k) - v_i(x^k) &=&  g_i(x^k) - g_i(x^*)+ q_i(x^k) - q_i(x^*) - [v_i(x^k)-v_i(x^*)], \\ [5pt]
&=& t_k [g'_i(x^*; d) + q'_i(x^*;d) - v'_i(x^*;d)] + o(t_k) \ < \ 0 
\ea
\]
when $k \gg 1$. Hence, $x^k$ is a feasible point when $k$ is sufficiently large. Using \eqref{obj-dir} 
and a similar argument as above, we have 
\[
f(x^k) + p(x^k) - u(x^k) < f(x^*) + p(x^*) - u(x^*)
\]
for all sufficiently large $k$. In addition, notice that $x^k \to x^*$ as $k \to \infty$.  These results imply 
that $x^*$ is not a local minimizer, which is a contradiction to the assumption. 

Case 2): Suppose that there exists some $i_0 \in \cA(x^*)$ such that 
\[
g'_{i_0}(x^*; d) + q'_{i_0}(x^*;d) - v'_{i_0}(x^*;d) = 0.
\]
It then together with \eqref{AT} implies that $i_0 \in \cA_\cT(x^*)$. By the assumption, there exists 
 $0 \neq \bd \in \cT_{\cX}(x^*)$ such that \eqref{slater} holds. Since $\bd \in \cT_{\cX}(x^*)$, 
 there exist a positive sequence $\{\eta_l\} \downarrow 0$ and 
a sequence $\{y^l\} \subseteq \cX$ such that $y^l = x^* + \eta_l \bd + o(\eta_l)$. Let 
$\bd^l = (y^l - x^*)/\eta_l$. Clearly, $\|\bd^l - \bd\| = o(1)$. It follows that for all $i$, 
\[
g'_i(x^*;\bd^l) - g'_i(x^*;\bd) + q'_i(x^*;\bd^l) - q'_i(x^*;\bd) - [\inf\limits_{s \in 
\partial v_i(x^*)} s^T\bd^l - \inf\limits_{s \in \partial v_i(x^*)} s^T\bd] 
= O (\|\bd^l - \bd\|) = o(1),
\]
which together with \eqref{slater} implies that for sufficiently large $l$,  
\beq \label{perturb-slater1}
g'_i(x^*;\bd^l) + q'_i(x^*;\bd^l) - \inf\limits_{s \in \partial v_i(x^*)} s^T\bd^l \ < \ 0, \quad 
\forall i \in \cA_\cT(x^*).  
\eeq
Let $\{\alpha_l\} \subset (0,1]$ be a sequence such that $\alpha_l  \downarrow 0$, and let 
\[
d^l = (1-\alpha_l) d + \alpha_l \bd^l.
\] 
Claim that for sufficiently large $l$,   
\beq \label{perturb-slater2}
g'_i(x^*; d^l) + q'_i(x^*;d^l) - v'_i(x^*;d^l)  \ < \ 0, \ \forall i \in \cA(x^*). 
\eeq 
Indeed,  we arbitrarily choose $i \in \cA(x^*)$. If $g'_i(x^*; d) + q'_i(x^*;d) - v'_i(x^*;d)  <  0$, 
we then have 
\[
\lim\limits_{l \to \infty }g'_i(x^*; d^l) + q'_i(x^*;d^l) - v'_i(x^*;d^l) = g'_i(x^*; d) + q'_i(x^*;d) - v'_i(x^*;d)  <  0,
\] 
which immediately implies that \eqref{perturb-slater2} holds for sufficiently large $l$. We now 
suppose that 
\beq \label{0-dir}
g'_i(x^*; d) + q'_i(x^*;d) - v'_i(x^*;d)  = 0.
\eeq
Hence, $i \in \cA_\cT(x^*)$. Let $s^* \in \Arg\max\limits_s\{s^Td: s \in \partial v_i(x^*)\}$. Using 
\eqref{perturb-slater1}, \eqref{0-dir}, convexity, and the definition of $\{d^l\}$, we have
\[
\ba{l}
g'_i(x^*; d^l) + q'_i(x^*;d^l) - v'_i(x^*;d^l)  \\ [8pt]
\le (1-\alpha_l) g'_i(x^*; d) + \alpha_l g'_i(x^*;\bd^l) + (1-\alpha_l) q'_i(x^*; d) + \alpha_l q'_i(x^*;\bd^l) - 
(s^*)^T [(1-\alpha_l) d + \alpha_l \bd^l] \\ [8pt]
= (1-\alpha_l)[g'_i(x^*; d) + q'_i(x^*;d) - v'_i(x^*;d)] + 
\alpha_l [g'_i(x^*;\bd^l) + q'_i(x^*;\bd^l) - (s^*)^T \bd^l ] \\ [8pt] 
= \alpha_l [g'_i(x^*;\bd^l) + q'_i(x^*;\bd^l) - (s^*)^T \bd^l] \ \le \ \alpha_l [g'_i(x^*;\bd^l) + q'_i(x^*;\bd^l) 
- \inf\limits_{s \in \partial v_i(x^*)} s^T\bd^l] \ < \ 0,  
\ea
\]
and hence \eqref{perturb-slater2} again holds for sufficiently large $l$.  Now let the sequence $\{x^{k,l}\}$ be defined as 
\beq \label{xkl}
x^{k,l} = (1-\alpha_l)x^k + \alpha_l(x^*+t_k \bd^l), \ \forall k, l \ge 1.
\eeq
By the definition of $\bd^l$, one can observe that $x^*+t_k \bd^l \in \cX$ for sufficiently 
large $k$. It then follows that for each $l$,  $\tx^{k,l} \in \cX$ when $k \gg 1$ due to 
$x^k \in \cX$ and convexity of $\cX$. Recall that $x^k = x^* + t_k d + o(t_k)$, which together with 
\eqref{xkl} yields 
\[
x^{k,l} =  x^* + t_k d^l + o(t_k).
\]
Using this relation and \eqref{perturb-slater2}, one can obtain that, for any $i\in\cA(x^*)$ 
and sufficiently large $l$,  
\[
\ba{lcl}
g_i(x^{k,l}) + q_i(x^{k,l}) - v_i(x^{k,l}) &=&  g_i(x^{k,l}) - g_i(x^*)+ q_i(x^{k,l}) - q_i(x^*) - 
[v_i(x^{k,l})-v_i(x^*)], \\ [5pt]&=& t_k [g'_i(x^*; d^l) + q'_i(x^*;d^l) - v'_i(x^*;d^l)] + o(t_k) 
\ < 0 
\ea
\]
whenever $k \ge n_l$ for some sequence $\{n_l\}$. Hence, $x^{k,l}$ is a feasible point for $k \ge n_l$ 
and sufficiently large $l$. Using \eqref{obj-dir} and the fact $d^l \to d$ as $l \to \infty$, we know that 
\[
f'(x^*; d^l) + p'(x^*;d^l) - u'(x^*;d^l)  \ < \ 0.
\]
Using this relation and a similar argument as above, we obtain that for sufficiently large $l$, 
\[
f(x^{k,l}) + p(x^{k,l}) - u(x^{k,l}) < f(x^*) + p(x^*) - u(x^*)
\]
whenever $k \ge \bar n_l $ for some sequence $\{\bar n_l\}$. Notice that $x^{k,l} \to x^*$ as $k,l \to \infty$. 
The above results again contradicts with the assumption that $x^*$ is a local minimizer. Therefore, 
$A \cap B \neq \emptyset$.  The conclusion of this theorem then immediately follows from this relation 
and the definitions of $A$ and $B$.
\end{proof}

\gap

\begin{remark}
\bi
\item[(a)] Condition \eqref{closedness} is satisfied if $\cX$ is a polyhedron and   
$\sum\limits_{i \in \cA(x^*)} \cone(\nabla g_i(x^*) + \partial q_i(x^*) - \partial v_i(x^*))$ 
is a finitely generated cone or if
\[
\left(-\sum\limits_{i \in \cA(x^*)} \cone(\nabla g_i(x^*) + \partial q_i(x^*) - \partial v_i(x^*))\right)
\cap \cN_{\cX}(x^*) = \{0\}.
\]
It thus follows that, if $\cX$ is a polyhedron and $q_i$ and $v_i$ are differentiable or piecewise 
convex functions (e.g., $\|x\|_1$) for each $i \in \cA(x^*)$, condition \eqref{closedness} holds.
\item[(b)] When $f$ and $g_i$'s are convex, condition \eqref{slater} holds if there exists a generalized 
Slater point $\bx \in \cX$, that is, $\bx$ satisfies 
\[
g_i(\bx) + q_i(\bx) - v_i(x^*) - \inf\limits_{s \in \partial v_i(x^*)} s^T(\bx -x^*) \ < \ 0, \quad 
\forall i \in \cA(x^*).  
\]
Indeed, let $\bd=\bx-x^*$. Clearly, $\bd \in \cT_{\cX}(x^*)$. Moreover, for each $i \in \cA(x^*)$,
\[
\ba{lcl}
g'_i(x^*; \bd) + q'_i(x^*;\bd) - \inf\limits_{s \in \partial v_i(x^*)} s^T\bd  &\le& g_i(x^*+\bd) - g_i(x^*) 
+ q_i(x^*+\bd) - q_i(x^*) + v_i(x^*) \\ [8pt]
& & - \ v_i(x^*) - \inf\limits_{s \in \partial v_i(x^*)} s^T(\bx -x^*), \\ [8pt]
&=& g_i(\bx) + q_i(\bx) - v_i(x^*) - \inf\limits_{s \in \partial v_i(x^*)} s^T(\bx -x^*) \ < \ 0,
\ea
\]
and hence condition \eqref{slater} holds.
\ei
\end{remark}

\section{A sequential convex programming method} 
\label{exact}

In this section we propose a sequential convex programming (SCP) method for solving 
problem \eqref{general-dc} in which each iteration is obtained by solving a convex 
 programming problem. We also propose a variant of it for solving \eqref{general-dc}. Before 
proceeding, we introduce some notations that will be used subsequently.

%

For each $x\in \cX$, $s_f$, $s_u$, $s_{g_i}$, $s_{v_i} \in \Re^n$ for $i=1,\ldots, m$, 
we define 
\beqa 
& \cC(x,\{s_{g_i}\}^m_{i=1},\{s_{v_i}\}^m_{i=1}) 
= \left\{y\in \cX: \ba{l} g_i(x)+s^T_{g_i} (y-x) + \frac{L_{g_i}}{2}\|y-x\|^2+q_i(y)  \\
      - [v_i(x)+s^T_{v_i}(y-x)] \ \le \ 0 
\ea 
\right\}, \label{set-c} \\ [6pt]
& h(y;x,s_f,s_u) = f(x)+s^T_{f} (y-x) + \frac{L_{f}}{2}\|y-x\|^2+p(y)
      - [u(x)+s^T_{u}(y-x)]. \nn
\eeqa
In addition, we denote by $\cF$ the feasible region of problem \eqref{general-dc}. 

\vgap


We are now ready to present an SCP method for solving problem \eqref{general-dc}. 

\gap

\noindent
{\bf Exact sequential convex programming method for \eqref{general-dc}:}  \\ [5pt]
Let $x^0 \in \cF$ be arbitrarily chosen. Set $k=0$. 
\begin{itemize}
\item[1)] Compute $s^k_f = \nabla f(x^k)$, $s^k_u\in\partial u(x^k)$, $s^k_{g_i} = \nabla g_i(x^k)$, 
$s^k_{v_i} \in \partial v_i(x^k)$ for all $i$.
\item[2)] Solve
\beq \label{CP}
x^{k+1} \in \Arg\min\limits_y \{h(y;x^k,s^k_f,s^k_u): y \in \cC(x^k,\{s^k_{g_i}\}^m_{i=1},\{s^k_{v_i}\}^m_{i=1})\}.
\eeq 
\item[3)] Set $k \leftarrow k+1$ and go to step 1). 
\end{itemize}
\noindent
{\bf end}

\gap

\begin{remark}
\bi
\item[(a)] When $\cX = \Re^n$, $p\equiv 0$, $u\equiv 0$, $L_f >0$, $q_i\equiv 0$, $v_i \equiv 0$ 
and $L_{g_i}>0$ for all $i$, the above method becomes the  MBA method proposed in \cite{AuShTe10}.
\item[(b)] When $m=1$, $f\equiv 0$, $g_1 \equiv 0$, and $p$, $u$, $q_1$, $u_1$ are smooth convex 
functions in $\cX$, the above method becomes the method studied in \cite{HoYaZh11}. 
\item[(c)] When $m=0$ and $f \equiv 0$, the above method becomes the well-known method
 \cite{TaAn98,HoTh99} for DC programming.
\ei
\end{remark}

In what follows, we will establish that under some assumptions, any accumulation point of the sequence 
$\{x^k\}$ generated above is a KKT point of problem \eqref{general-dc}.  Before proceeding, we state 
several lemmas that will be used subsequently.

The following lemma is well known (see, for example, \cite{OrRh00}), which provides an upper bound 
for a smooth function with Lipschitz continuous gradient. 

\begin{lemma} \label{upper-bound}
Let $\Omega \subseteq \Re^n$ be a closed convex set, and $h$ a differentiable function in $\Omega$. 
Suppose that there exists some constant $L_h \ge 0 $ such that  
\[
\|\nabla h(x) - \nabla h(y)\| \le L_h \|x-y\|, \ \ \ \forall x, y \in \Omega.
\]
Then, for any $L \ge L_h$, 
\[
h(y) \ \le \ h(x) + \nabla h(x)^T(y-x) + \frac{L}{2} \|y-x\|^2, \ \ \ \forall x,y \in \Omega.
\]
\end{lemma}

\vgap

The following lemma is due to Robinson \cite{Rob75}, which provides an error bound for a class of 
convex inequalities.

\begin{lemma} \label{err-bdd}
Let $X$ be a closed convex set in $\Re^n$, and $\cK$ a nonempty closed convex cone in $\Re^m$. 
Suppose that $g: X \to \Re^m$ is a $\cK$-convex function, that is, 
\[
\lambda g(x^1) + (1-\lambda) g(x^2) \in g(\lambda x^1 + (1-\lambda) x^2) + \cK.
\]
Assume that $x^s \in X$ is a generalized Slater point for the set $\Omega :=\{x\in X: 0 \in g(x) + \cK\}$, 
that is, there exists $\delta>0$ such that $\cB(0; \delta) \subseteq g(x^s) + \cK$, where $\cB(0;\delta)$ 
is the closed ball centered at $0$ with radius $\delta$. Then, 
\[
\dist(x,\Omega) \ \le \ \delta^{-1}\|x-x^s\| \dist(0, g(x)+\cK), \ \ \ \forall x \in X.
\]
\end{lemma}

\vgap

The following lemma states a simple property of the set $\cC$ that is defined in \eqref{set-c}.

\begin{lemma} \label{set-C}
For each $x\in \cF$, let $s_{g_i} = \nabla g_i(x)$ and $s_{v_i}\in\partial v_i(x)$. Then, 
$\cC(x,\{s_{g_i}\}^m_{i=1},\{s_{v_i}\}^m_{i=1})$ is a nonempty closed convex set in $\cF$.
\end{lemma}

\begin{proof}
Since $x\in\cF$, one can clearly see that $x\in \cC(x,\{s_{g_i}\}^m_{i=1},\{s_{v_i}\}^m_{i=1})$. Hence,
 $\cC(x,\{s_{g_i}\}^m_{i=1}, \\ \{s_{v_i}\}^m_{i=1}) \neq \emptyset$.  Due to $s_{v_i}\in\partial v_i(x)$, 
we know that $v_i(y) \ge v_i(x) + s^T_{v_i} (y-x), \forall y \in \Re^n$.
Using this relation and Lemma \ref{upper-bound}, one can see that for any 
$y\in \cC(x,\{s_{g_i}\}^m_{i=1},\{s_{v_i}\}^m_{i=1})$, $y$ is in $\cX$ and
$g_i(y) + q_i(y) -v_i(y) \le 0$ for $i=1, \ldots, m$. Hence, $y\in \cF$. It 
implies that $\cC(x,\{s_{g_i}\}^m_{i=1},\{s_{v_i}\}^m_{i=1}) \subseteq \cF$. 
Finally, it is easy to see that $\cC(x,\{s_{g_i}\}^m_{i=1},\{s_{v_i}\}^m_{i=1})$ 
is a closed convex set. 
\end{proof}

\gap

We are now  ready to establish that under some assumptions, any accumulation point of the sequence 
$\{x^k\}$ generated by the above SCP method is a KKT point of problem \eqref{general-dc}.

\begin{theorem} \label{exact-thm}
Let $\{(x^k,s^k_f,s^k_u,\{s^k_{g_i}\}^m_{i=1},\{s^k_{v_i}\}^m_{i=1})\}$ be the sequence 
generated by the above SCP method. The following statements hold:
\bi
\item[\rm (i)] $\{x^k\} \subset \cF$ and $\{f(x^k)+p(x^k)-u(x^k)\}$ is monotonically nonincreasing.
\item[\rm (ii)] Suppose further that  $(x^*,s^*_f,s^*_u,\{s^*_{g_i}\}^m_{i=1},
\{s^*_{v_i}\}^m_{i=1})$ is an accumulation point of $\{(x^k,s^k_f,s^k_u,\\\{s^k_{g_i}\}^m_{i=1},
 \{s^k_{v_i}\}^m_{i=1})\}$. Assume that Slater's condition holds for the set 
$\cC(x^*,\{s^*_{g_i}\}^m_{i=1},\{s^*_{v_i}\}^m_{i=1})$, that is, there exists $\by \in \cX$ such that 
\beq \label{slater-cond}
g_i(x^*)+(s^*_{g_i})^T (\by-x^*) + \frac{L_{g_i}}{2}\|\by-x^*\|^2+q_i(\by) 
      - [v_i(x^*)+(s^*_{v_i})^T(\by-x^*)] \ < \ 0,  \ i=1,\ldots, m.
\eeq 
Then, $x^*$ is a KKT point of problem \eqref{general-dc}.
\ei
\end{theorem}

\begin{proof}
(i) We know that $x^0 \in \cF$.  Since $x^1 \in \cC(x^0,\{s^0_{g_i}\}^m_{i=1},\{s^0_{v_i}\}^m_{i=1})$, 
it follows from Lemma \ref{set-C} that $x^1 \in \cF$. By repeating this argument, we can conclude 
that $\{x^k\} \subset \cF$. In addition, notice that $x^k \in \cC(x^k,\{s^k_{g_i}\}^m_{i=1},\{s^k_{v_i}\}^m_{i=1})$. 
Hence, we have
\[
h(x^{k+1};x^k,s^k_f,s^k_u) \ \le \ h(x^{k};x^k,s^k_f,s^k_u) \ = \ f(x^k)+p(x^k)-u(x^k).
\] 
Since $s^k_u \in\partial u(x^k)$, we know that $u(x^{k+1}) \ge u(x^k) + (s^k_u)^T (x^{k+1}-x^k)$. 
Using this relation and Lemma \ref{upper-bound}, one can see that 
\[
f(x^{k+1}) + p(x^{k+1}) - u(x^{k+1})  \ \le \ h(x^{k+1};x^k,s^k_f,s^k_u).
\]
It then follows that 
\beq \label{f-ineq}
f(x^{k+1})+p(x^{k+1})-u(x^{k+1}) \ \le \ h(x^{k+1};x^k,s^k_f,s^k_u) \ \le \ f(x^k)+p(x^k)-u(x^k).
\eeq
Thus, $\{f(x^k)+p(x^k)-u(x^k)\}$ is monotonically nonincreasing. 

(ii) Let $w := (x,\{s_{g_i}\}^m_{i=1},\{s_{v_i}\}^m_{i=1})$, $w^k := (x^k,\{s^k_{g_i}\}^m_{i=1},
\{s^k_{v_i}\}^m_{i=1})$, $w^* := (x^*,\{s^*_{g_i}\}^m_{i=1},\{s^*_{v_i}\}^m_{i=1})$. By the assumption, 
there exists a subsequence $K$ such that $\{(s^k_f,s^k_u,w^k)\}_K \to (s^*_f,s^*_u,w^*)$. We first show 
that for any $z \in \cC(w^*)$, there exists $z^k \in \cC(w^k)$ such that $\{z^k\}_K \to z$, where $\cC$ 
is defined in \eqref{set-c}. Indeed, let 
\[
\cG_i(y,w) := g_i(x)+s^T_{g_i} (y-x) + \frac{L_{g_i}}{2}\|y-x\|^2+q_i(y) - [v_i(x)+s^T_{v_i}(y-x)] \ \ \ \forall i,
\] 
and $\cG(y,w) := (\cG_1(y,w), \ldots, \cG_m(y,w))$. It follows from \eqref{slater-cond} that $\cG(\by,w^*) <0$. 
Hence, there exists $\delta >0$ such that 
\beq \label{slater-ws}
\cB(0; \delta) \subseteq \cG(\by,w^*) + \Re^m_+.
\eeq
Notice that $\cG(\by,w)$ is continuous in $w$ and $\{w^k\}_K \to w^*$. Hence,  when $k \in K $ is sufficiently large, 
$\|\cG(\by,w^k) - \cG(\by,w^*)\| \le \delta/2$ holds. It immediately implies that, for sufficiently large $k\in K$, 
\[
\cG(\by,w^*) - \cG(\by,w^k) + \cB(0; \delta/2) \subseteq \cB(0; \delta).
\]
This relation together with \eqref{slater-ws} yields that, for sufficiently large $k\in K$, 
\[
\ba{lcl}
\cG(\by,w^k) + \Re^m_+  &=&  \cG(\by,w^k) - \cG(\by,w^*) + 
\cG(\by,w^*) + \Re^m_+ \ \supseteq \  \cG(\by,w^k) - \cG(\by,w^*) +  \cB(0; \delta) \\
&\supseteq & \cG(\by,w^k) - \cG(\by,w^*)  + [\cG(\by,w^*) - \cG(\by,w^k) + \cB(0; \delta/2)] 
\ = \  \cB(0; \delta/2). 
\ea
\]
Hence, $\by$ is also a generalized Slater point for the set $\cC(w^k)$ when $k\in K$ is sufficiently 
large. In addition, it is not hard to verify that $\cG(y,w^k)$ is $\Re^m_+$-convex. Letting 
$g(\cdot)=\cG(\cdot,w^k)$, $\cK=\Re^m_+$, $\Omega = \cC(w^k)$, $X=\cX$, and using Lemma \ref{err-bdd}, 
we obtain that, for sufficiently large $k\in K$, 
\beq \label{dist-bdd}
\dist(y,\cC(w^k)) \ \le \ 2\delta^{-1}\|y-\by\| \dist(0, \cG(y,w^k)+\Re^m_+), \ \ \ \forall y \in \cX.
\eeq
Let $z \in \cC(w^*)$ be arbitrarily given, and  let $z^k = \arg\min\limits_y\{\|z-y\|: y \in \cC(w^k)\}$. 
Notice that $z \in \cX$. It then follows from \eqref {dist-bdd} with $y=z$ that,  when $k\in K$ is sufficiently large,
\[
\|z^k-z\| = \dist(z,\cC(w^k)) \ \le \ 2\delta^{-1}\|z-\by\| \dist(\cG(z,w^k),-\Re^m_+).
\]
Since $z\in \cC(w^*)$, we can observe that  $\{\dist(\cG(z,w^k),-\Re^m_+)\}_K \to \dist(\cG(z,w^*),-\Re^m_+) =0$. Using 
this relation and the above inequality, we obtain that $\{z^k\}_K \to z$ and $z^k \in \cC(w^k)$.

Since $\{x^k\}_K \to x^*$, by continuity we have $\{f(x^k)+p(x^k)-u(x^k)\}_K \to f(x^*)+p(x^*)-u(x^*)$. Notice that 
$\{f(x^k)+p(x^k)-u(x^k)\}$ is monotonically nonincreasing. Hence, we have $f(x^k)+p(x^k)-u(x^k) \to f(x^*)+p(x^*)-u(x^*)$, 
which together with \eqref{f-ineq} implies that $h(x^{k+1};x^k,s^k_f,s^k_u) \to f(x^*)+p(x^*)-u(x^*)$. 
Recall that $x^{k+1} \in \Arg\min\{h(y;x^k,s^k_f,s^k_u): y \in \cC(w^k)\}$. Since $z^k \in \cC(w^k)$, 
we obtain that $h(x^{k+1};x^k,s^k_f,s^k_u) \le h(z^k;x^k,s^k_f,s^k_u)$. Upon taking limits on both sides 
of this inequality as $k\in K \to \infty$, we have 
\[
f(x^*)+p(x^*)-u(x^*) \le h(z;x^*,s^*_f,s^*_u), \ \ \ \forall z \in \cC(w^*).
\]
In addition, since $\{x^k\} \subset \cF$ and $\{x^k\}_K \to x^*$, we know that $x^* \in \cF$, which 
yields $x^* \in \cC(w^*)$. Also,  $f(x^*)+p(x^*)-u(x^*)= h(x^*;x^*,s^*_f,s^*_u)$. Therefore, 
\beq \label{lim-pt}
x^* \in \Arg\min\{ h(z;x^*,s^*_f,s^*_u): z \in \cC(w^*)\}.
\eeq
Since Slater's condition holds for $\cC(w^*)$,  the first-order optimality condition of 
\eqref{lim-pt} immediately implies that $x^*$ is a KKT point of \eqref{general-dc}.
\end{proof}

\gap

\begin{remark}
Since $s^k_f = \nabla f(x^k)$, $s^k_u\in\partial u(x^k)$, $s^k_{g_i} = \nabla g_i(x^k)$, 
and $s^k_{v_i} \in \partial v_i(x^k)$ for all $i$, we observe that if $\{x^k\}$ has an 
accumulation point,  so is $\{s^k_f,s^k_u,\{s^k_{g_i}\}^m_{i=1},\{s^k_{v_i}\}^m_{i=1})\}$. Therefore, 
the first assumption in statement (ii) is mild. We next provide a sufficient condition for 
the second assumption to hold. In particular, we show that the assumption \eqref{slater-cond} 
holds if  the following generalized Mangasarian-Fromovitz constraint qualification (MFCQ) holds at $x^*$.
\end{remark}

\gap

\begin{proposition} \label{suff-cond}
Let $x^*$ be a point in $\cF$. If the generalized MFCQ holds at $x^*$, that is, $\exists d \in \cT_{\cX}(x^*)$ such 
that 
\beq\label{MFCQ}
g'_i(x^*;d) + q'_i(x^*;d) - \inf\limits_{s \in \partial v_i(x^*)} s^T d \ < \ 0, \quad 
\forall i \in \cA(x^*).
\eeq
Then, \eqref{slater-cond} holds at $x^*$ for $s^*_{g_i} = \nabla g_i(x^*)$ and every $s^*_{v_i}\in\partial v_i(x^*)$.
\end{proposition}

\begin{proof}
Let $d$ be given above, $s^*_{g_i} = \nabla g_i(x^*)$ and $s^*_{v_i}\in\partial v_i(x^*)$.  
Then, there exist a positive sequence $\{t_k\} \downarrow 0$ and a sequence $\{x^k\} \subseteq \cX$ such that 
$x^k = x^* + t_k d + o(t_k)$. For each $i\in\cA(x^*)$, we have that, for sufficiently large $k$, 
\[
\ba{l}
g_i(x^*)+(s^*_{g_i})^T (x^k-x^*) + \frac{L_{g_i}}{2}\|x^k-x^*\|^2+q_i(x^k) 
      - [v_i(x^*)+(s^*_{v_i})^T(x^k-x^*)] \\ [5pt]
= (s^*_{g_i})^T (x^k-x^*) + \frac{L_{g_i}}{2}\|x^k-x^*\|^2+q_i(x^k)-q_i(x^*) 
      - (s^*_{v_i})^T(x^k-x^*) \\ [5pt]
= t_k (s^*_{g_i})^Td + q_i(x^*+t_k d)-q_i(x^*) - t_k(s^*_{v_i})^T d + o(t_k) \\ [8pt]
= t_k [(s^*_{g_i})^Td + q'_i(x^*;d) - (s^*_{v_i})^T d + o(1)]  \ \le \  t_k [g'_i(x^*;d) 
+ q'_i(x^*;d) - \inf\limits_{s \in \partial v_i(x^*)} s^T d  + o(1)] \ < \ 0,      
\ea      
\]   
where the last inequality follows from \eqref{MFCQ}. In addition, for each $i \notin \cA(x^*)$, 
we know that $g_i(x^*)+q_i(x^*)-v_i(x^*) <0$. Notice that $x^k \to x^*$ as $k \to \infty$. Hence, for sufficiently 
large $k$, we have 
\[
g_i(x^*)+(s^*_{g_i})^T (x^k-x^*) + \frac{L_{g_i}}{2}\|x^k-x^*\|^2+q_i(x^k) 
      - [v_i(x)+(s^*_{v_i})^T(x^k-x^*)] < 0, \ \ \ i=1,\ldots, m.
\]
\end{proof}

\gap

The above SCP method uses the global Lipschitz constants of $\nabla f$ and $\nabla g_i$'s, which may be too 
conservative. To improve its practical performance, we can use  ``local'' Lipschitz constants that are updated 
dynamically. In addition, the above method is a monotone method since $\{f(x^k)+p(x^k)-u(x^k)\}$ is nonincreasing. 
As mentioned in \cite{GrLaLu86,BiMaRa00,WrNoFi09,LuZh12}, nonmonotone methods generally outperform monotone counterparts 
for many nonlinear programming problems. We next propose a variant of the SCP in which ``local'' Lipschitz 
constants and nonmonotone scheme are used. Before proceeding, we introduce some notations as 
follows.

For each $x\in \cF$, $l_f$, $l_{g_i} \in \Re$, $s_f$, $s_u$, $s_{g_i}$, $s_{v_i} \in \Re^n$ for $i=1,\ldots, m$, we define 
\beqa
\bC(x,\{l_{g_i}\}^m_{i=1},\{s_{g_i}\}^m_{i=1},\{s_{v_i}\}^m_{i=1}) 
&=& \left\{y\in \cX: \ba{l} g_i(x)+s^T_{g_i} (y-x) + \frac{l_{g_i}}{2}\|y-x\|^2+q_i(y)  \\
      - [v_i(x)+s^T_{v_i}(y-x)] \ \le \ 0
\ea 
\right\},  \label{set-bc} \\ [5pt]
 \bh(y;x,l_f,s_f,s_u) &=& f(x)+s^T_{f} (y-x) + \frac{l_{f}}{2}\|y-x\|^2+p(y)
      - [u(x)+s^T_{u}(y-x)], \nn \\ [5pt] 
 F(x) &:=& f(x) + p(x) - u(x). \nn
\eeqa


\vgap

We are now ready to  present a variant of the above SCP method.

\gap

\noindent
{\bf A variant of SCP method for \eqref{general-dc}:}  \\ [5pt]
Choose parameters $c>0$, $0< L_{\min} < L_{\max}$, $\tau>1$, and integer $M \ge 0$. Set $k=0$ and choose an arbitrary $x^0 \in \cF$. 
\begin{itemize}
\item[1)] Compute $s^k_f = \nabla f(x^k)$, $s^k_u\in\partial u(x^k)$, $s^k_{g_i} = \nabla g_i(x^k)$, 
$s^k_{v_i} \in \partial v_i(x^k)$ for all $i$.
\item[2)] Choose $l^{k,0}_f$, $l^{k,0}_{g_i} \in [L_{\min}, L_{\max}]$ arbitrarily, and set $l^k_f = l^{k,0}_f$ and $l^{k}_{g_i} = l^{k,0}_{g_i}$ for all $i$.  
\item[3)] Find 
\beq \label{CP-lk}
x^{k+1} = \arg\min\limits_y \{\bh(y;x^k,l^k_f,s^k_f,s^k_u): y \in \bC(x^k,\{l^k_{g_i}\}^m_{i=1},\{s^k_{g_i}\}^m_{i=1},\{s^k_{v_i}\}^m_{i=1})\}.
\eeq 
\bi 
\item[3a)] If $x^{k+1} \in \cF$ and 
\beq \label{reduct} 
F(x^{k+1}) \le \max_{[k-M]^+ \le i \le k} F(x^i) - \frac{c}{2} \|x^{k+1}-x^k\|^2
\eeq
holds, go to step 4).
\item[3b)] If $x^{k+1} \notin \cF$, set $l^k_{g_i} \leftarrow \tau l^k_{g_i}$ for all $i$ and go to step 3).
\item[3c)] If \eqref{reduct} does not hold, set $l^k_f \leftarrow \tau l^k_f$ and go to step 3). 
\ei
\item[4)]
Set $k \leftarrow k+1$ and go to step 1). 
\end{itemize}
\noindent
{\bf end}

\vgap

\begin{remark}
\bi
\item[(i)] When $M=0$, the above method becomes a monotone method.
\item[(ii)]
In practical computation, $l^{k,0}_f$, $l^{k,0}_{g_i}$ can be updated by the similar strategy as used in 
\cite{BaBo88,BiMaRa00}, that is, 
\[
\ba{lcl}
l^{k,0}_f& =& \max\left\{L_{\min},\min\left\{L_{\max},\frac{\Delta x^T \Delta f}{\|\Delta x\|^2}\right\}\right\}, \\ [6pt]
l^{k,0}_{g_i} &=& \max\left\{L_{\min},\min\left\{L_{\max},\frac{\Delta x^T \Delta g_i}{\|\Delta x\|^2}\right\}\right\},  \ \ \ \forall i,
\ea
\]
where $\Delta x = x^k -x^{k-1}$, $\Delta f= \nabla f(x^k) - \nabla f(x^{k-1})$, and 
$\Delta g_i= \nabla g_i(x^k) - \nabla g_i(x^{k-1})$ for all $i$.
\item[(iii)] 
$l^k_f$ and $\{l^k_{g_i}\}^m_{i=1}$ can be updated by some other strategies. For example,  
\bi
\item[1)]
we may update $l^k_f$ and $\{l^k_{g_i}\}^m_{i=1}$ simultaneously, that is, steps 3b) and 3c) can be replaced by:
{\center
if $x^{k+1} \notin \cF$ or \eqref{reduct} does not hold, set $l^k_f \leftarrow \tau l^k_f$ and $l^k_{g_i} \leftarrow \tau l^k_{g_i}$ for all $i$;} 
\item[2)] in step 3b), each $l^k_{g_i}$ can be updated individually. In particular, for each $i$, 
we  can update $l^k_{g_i}$ only if the $i$th constraint of \eqref{general-dc} is violated at $x^{k+1}$, 
that is, $g_i(x^{k+1}) +q_i(x^{k+1})-v_i(x^{k+1})>0$.   
\ei
\ei
\end{remark}

\gap

We first show that for each outer iteration, its number of inner iterations 
is finite.  

\begin{theorem} \label{inner-complexity}
At each $k$th outer iteration, its associated inner iterations terminate after at most 
\beq \label{sum-inner}
\left\lfloor \frac{\log(L_f+c)+\log(\max\limits_i L_{g_i})-2\log(2L_{\min})}
{\log \tau} +4\right\rfloor
\eeq
loops.
\end{theorem}

\begin{proof}
Let $\bar l^k_f$ and $\bar l^k_{g_i}$ denote the final value of $l^k_f$ and $l^k_{g_i}$ at the $k$th outer 
iteration, respectively. Note that $\bh(\cdot;x^k,l^k_f,s^k_f,s^k_u)$ is a strongly convex function with 
modulus $l^k_f>0$. It then follows from \eqref{CP-lk} that 
\[
F(x^k) = f(x^k) + p(x^k) - u(x^k) = \bh(x^k;x^k,l^k_f,s^k_f,s^k_u) \ \ge \  \bh(x^{k+1};x^k,l^k_f,s^k_f,s^k_u) 
+ \frac{l^k_f}{2} \|x^{k+1}-x^k\|^2.
\]
Since $s^k_u \in\partial u(x^k)$, we know that $u(x^{k+1}) \ge u(x^k) + (s^k_u)^T (x^{k+1}-x^k)$. 
Using this relation and Lemma \ref{upper-bound}, one can see that 
\[
F(x^{k+1})  \ = \ f(x^{k+1}) + p(x^{k+1}) - u(x^{k+1})  \ \le \ \bh(x^{k+1};x^k,l^k_f,s^k_f,s^k_u) + \frac{L_f-l^k_f}{2}
\|x^{k+1}-x^k\|^2.
\]
The above two inequalities yield 
\[ 
F(x^{k+1})  \ \le \ \ F(x^k)  - (l^k_f-\frac{L_f}{2})\|x^{k+1}-x^k\|^2_2
\ \le \  \max\limits_{[k-M]^+ \le i \le k} F(x^i) - (l^k_f-\frac{L_f}{2})\|x^{k+1}-x^k\|^2_2.
\]
Similarly, one can show that 
\[
g_i(x^{k+1})+q_i(x^{k+1})-v_i(x^{k+1}) \ \le \ g_i(x^k)+q_i(x^k)-v_i(x^k)  - (l^k_{g_i}-\frac{L_{g_i}}{2})\|x^{k+1}-x^k\|^2_2,  \ \forall i,
\]
which together with $x^k \in \cF$ implies that 
\[
g_i(x^{k+1})+q_i(x^{k+1})-v_i(x^{k+1}) \ \le  - (l^k_{g_i}-\frac{L_{g_i}}{2})\|x^{k+1}-x^k\|^2_2,  \ \forall i.
\]
Hence, $x^{k+1} \in \cF$ and \eqref{reduct} holds whenever $l^k_f \ge (L_f+c)/2$ and $\min\limits_i l^k_{g_i} 
\ge (\max\limits_i L_{g_i})/2$, which, together with 
the definitions of $\bar l_k$ and $\bar l^k_{g_i}$, implies that $\bar l_k/\tau < (L_f+c)/2$ 
and $\min\limits_i \bar l^k_{g_i}/\tau < (\max\limits_i L_{g_i})/2$,  that is, 
$\bar l_k <\tau(L_f+c)/2$ and $\min\limits_i \bar l^k_{g_i} < \tau(\max\limits_i L_{g_i})/2$. 
Let $n^k_f$ and $n^k_g$ denote the number of inner iterations for updating $l^k_f$ and $l^k_{g_i}$ 
at the $k$th outer iteration. Then, we have 
\[
\ba{l}
L_{\min} \tau^{n^k_f-1}  \ \le \ L^{k,0}_f \tau^{n^k_f-1} \ = \  \bar l^k_f  \ < \  \tau(L_f+c)/2, \\
L_{\min} \tau^{n^k_g-1}  \ \le \ (\min\limits_i L^{k,0}_{g_i} )\tau^{n^k_g-1} \ = \  \min\limits_i\bar l^k_{g_i}  
\ < \   \tau(\max\limits_i L_{g_i})/2.
\ea
\] 
Hence, the total number of inner iterations, $n^k_f+n^k_g$, is bounded above by the quantity given in 
\eqref{sum-inner} and the conclusion holds.
\end{proof}

\gap

We next establish that under some assumptions, any accumulation point of the sequence 
$\{x^k\}$ generated by the above variant of the SCP method is a KKT point of problem \eqref{general-dc}.

\begin{theorem} \label{v-exact-thm}
Let $\{(x^k,s^k_f,s^k_u,\{s^k_{g_i}\}^m_{i=1},\{s^k_{v_i}\}^m_{i=1})\}$ be the sequence 
generated by the above variant of the SCP method. Assume that $F(x) := f(x)+p(x)-u(x)$ is 
uniformly continuous in the level set $\cL=\{x\in \cF: F(x) \le F(x^0)\}$. Suppose that  $(x^*,l^*_f,
\{l^*_{g_i}\}^m_{i=1},s^*_f,s^*_u,\{s^*_{g_i}\}^m_{i=1}, \{s^*_{v_i}\}^m_{i=1})$ is an accumulation
 point of $\{(x^k,l^k_f,\{l^k_{g_i}\}^m_{i=1},s^k_f,s^k_u,\{s^k_{g_i}\}^m_{i=1},\{s^k_{v_i}\}^m_{i=1})\}$. 
Then the following statements hold:
\item[\rm (i)] $\|x^{k+1}-x^k\| \to 0$ and $f(x^k)+p(x^k)-u(x^k) \to f(x^*)+p(x^*)-u(x^*)$. 
\item[\rm (ii)] Suppose further that Slater's condition holds for the constraint 
set $\cC(x^*,\{l^*_{g_i}\}^m_{i=1},\{s^*_{g_i}\}^m_{i=1},\{s^*_{v_i}\}^m_{i=1})$, that is, there exists 
$\by \in \cX$ such that 
\beq \label{bslater-cond}
g_i(x^*)+(s^*_{g_i})^T (\by-x^*) + \frac{l^*_{g_i}}{2}\|\by-x^*\|^2+q_i(\by) 
      - [v_i(x^*)+(s^*_{v_i})^T(\by-x^*)] \ < \ 0,  \ i=1,\ldots, m.
\eeq 
Then, $x^*$ is a KKT point of problem \eqref{general-dc}.
\end{theorem}

\begin{proof}
(i) By the definition of $x^k$, we observe that $\{x^k\} \subseteq \cL$. 
Let $d^k:=x^{k+1}-x^k$, and $l(k)$ an integer between $[k-M]^+$ and $k$ such that  
\[
F(x^{l(k)}) = \max \{F(x^i): [k-M]^+ \le i \le k\}, \ \ \ \forall k \ge 0.
\]
It follows from \eqref{reduct} that $F(x^{k+1}) \le F(x^{l(k)})$ for all $k \ge 0$, 
which together with the definition of $l(k)$ implies that $\{F(x^{l(k)})\}$ is monotonically 
nonincreasing. Further, by continuity of $F$ and $\{x^k\}_K \to x^*$, we know that $\{F(x^k)\}_K \to F(x^*)$. 
This together with the fact $F(x^{l(k)}) \ge F(x^k)$ implies that $\{F(x^{l(k)})\}_K$ is bounded below.  Using 
this result and the monotonicity of $\{F(x^{l(k)})\}$, we see that $\{F(x^{l(k)})\}$ is bounded below. Hence, 
there exists some $F^* \in \Re$ such that 
\beq \label{F-limit}
\lim_{k\to \infty} F(x^{l(k)}) = F^*.
\eeq
We can prove by induction that the following limits 
hold for all $j \ge 1$: 
\beq \label{2-limits}
\lim_{k \to \infty} d^{l(k)-j} = 0, \ \ \ \ \lim_{k \to \infty} F(x^{l(k)-j}) 
= F^*.
\eeq
Indeed, replacing $k$ by $l(k)-1$ in \eqnok{reduct} and using the definition of $l(k)$, 
we obtain that 
\[
F(x^{l(k)}) \le F(x^{l(l(k)-1)}) -\frac{c}{2}\|d^{l(k)-1}\|^2, 
\]
which together with \eqnok{F-limit} implies that  
$\lim_{k \to \infty} d^{l(k)-1} = 0$. Using this relation, \eqnok{F-limit} and 
uniform continuity of $F$ in $\cL$, we have 
\[
\lim_{k \to \infty} F(x^{l(k)-1}) = \lim_{k \to \infty} F(x^{l(k)}-d^{l(k)-1}) 
= \lim_{k \to \infty} F(x^{l(k)}) = F^*. 
\]
Therefore, \eqnok{2-limits} holds for $j=1$. Now, we assume that \eqnok{2-limits} holds 
for $j$. We need to show that it also holds for $j+1$. Replacing $k$ by $l(k)-j-1$ in 
\eqnok{reduct} and using the definition of $l(k)$, we have 
\[
F(x^{l(k)-j}) \ \le \ F(x^{l(l(k)-j-1)}) - \frac{c}{2}\|d^{l(k)-j-1}\|^2,
\]
which, together with \eqnok{F-limit} and the induction assumption $\lim_{k \to \infty} 
F(x^{l(k)-j}) = F^*$, implies that $\lim_{k \to \infty} d^{l(k)-j-1} = 0$. 
Using this result, $\lim_{k \to \infty} F(x^{l(k)-j}) = F^*$ and uniform continuity 
of $F$ in $\cL$, we see that $\lim_{k \to \infty} F(x^{l(k)-j-1}) = F^*$. Hence, 
\eqnok{2-limits} holds for $j+1$. It then follows from the induction 
that \eqnok{2-limits}  holds for all $j \ge 1$. 
Further, by the definition of $l(k)$, we see that 
for $k \ge M+1$, $k-M-1 = l(k)-j$ for some $1 \le j \le M+1$, which together 
with the first limit in \eqnok{2-limits}, implies that $\lim_{k \to \infty} 
d^k = \lim_{k \to \infty} d^{k-M-1} = 0$. Additionally, we observe that 
\[
x^{l(k)} = x^{k-M-1} + \sum^{\bar l_k}_{j=1} d^{l(k)-j} \ \ \ \forall k \ge M+1,
\]
where $\bar l_k = l(k)-(k-M-1) \le M+1$. Using the above identity, \eqnok{2-limits}, 
and uniform continuity of $F$ in $\cL$, we see that $\lim_{k \to \infty} F(x^k) 
= \lim_{k \to \infty} F(x^{k-M-1}) = F^*$, which, together with $\{F(x^k)\}_K \to F(x^*)$, 
implies that $F(x^k) \to F(x^*)$.  Hence, the statement (i) holds.

(ii)  Let $w := (x,\{l_{g_i}\}^m_{i=1},
\{s_{g_i}\}^m_{i=1},\{s_{v_i}\}^m_{i=1})$, $w^k := (x^k,
\{l^k_{g_i}\}^m_{i=1},\{s^k_{g_i}\}^m_{i=1},\{s^k_{v_i}\}^m_{i=1})$, $w^* := (x^*,\{l^*_{g_i}\}^m_{i=1},
\{s^*_{g_i}\}^m_{i=1},\{s^*_{v_i}\}^m_{i=1})$. By the assumption, 
there exists a subsequence $K$ such that $\{(l^k_f,s^k_f,s^k_u,w^k)\}_K \to (l^*_f,s^*_f,s^*_u,w^*)$. 
We first show that for any $z \in \bC(w^*)$, there exists 
$z^k \in \bC(w^k)$ such that $\{z^k\}_K \to z$, where $\bC$ is defined in \eqref{set-bc}. Indeed, 
let
\[
\bcG_i(y,w) := g_i(x)+s^T_{g_i} (y-x) + \frac{l_{g_i}}{2}\|y-x\|^2+q_i(y) - [v_i(x)+s^T_{v_i}(y-x)] \ \ \ \forall i,
\] 
and $\bcG(y,w) := (\bcG_1(y,w), \ldots, \bcG_m(y,w))$. Notice that $\bcG(\by,w)$ is continuous in $w$. 
Using this fact, \eqref{bslater-cond}, Lemma \ref{err-bdd}, and the similar arguments as in the proof of 
Theorem \ref{exact-thm} (ii), one can show that there exists some $\delta>0$ such that for  sufficiently large $k\in K$, 
\beq \label{dist-bdd-v}
\dist(y,\bC(w^k)) \ \le \ 2\delta^{-1}\|y-\by\| \dist(0, \bcG(y,w^k)+\Re^m_+), \ \ \ \forall y \in \cX.
\eeq
 Let $z \in \bC(w^*)$ be arbitrarily given, and  let $z^k = \arg\min\limits_y\{\|z-y\|: y \in \bC(w^k)\}$. 
Clearly, $z \in \cX$ and $\dist(\bcG(z,w^*),-\Re^m_+)=0$. Using these facts and letting $y=z$ in \eqref{dist-bdd-v}, one can obtain that $\{z^k\}_K \to z$ and $z^k \in \bC(w^k)$.

Recall from statement (i) that $\|x^{k+1}-x^k\| \to 0$. Since $\{x^k\}_K \to x^*$, it then 
follows that $\{x^{k+1}\}_K \to x^*$. Let $\bar l^k_f$ denote
 the final value of $l^k_f$ at the $k$th outer iteration. From the proof of Theorem 
\ref{inner-complexity}, we know that $\bar l^k \in [L_{\min}, \tau(L_f+c)/2]$. Using these facts and 
$\{F(x^k)\} \to F(x^*)$, we observe that 
\[
\{\bh(x^{k+1};x^k,\bar l^k_f,s^k_f,s^k_u)\}_K \to F(x^*).
\]
Recall that $x^{k+1} = \arg\min\{\bh(y;x^k,\bar l^k_f,s^k_f,s^k_u): y \in \bC(w^k)\}$. Since $z^k \in \bC(w^k)$, 
we have $\bh(x^{k+1};x^k,\bar l^k,s^k_f,s^k_u) \le \bh(z^k;x^k,\bar l^k,s^k_f,s^k_u)$. Upon taking 
limits on both sides of this inequality as $k\in K \to \infty$, we obtain that 
\[
F(x^*) \le \bh(z;x^*,l^*_f,s^*_f,s^*_u), \ \ \ \forall z \in \bC(w^*).
\]
In addition, we know that $x^* \in \cF$, which implies that $x^* \in \bC(w^*)$. Also,  
$F(x^*)= \bh(x^*;x^*,l^*_f,s^*_f,s^*_u)$. Hence, we have  
\beq \label{lim-pt-v}
x^* \in \Arg\min\{\bh(z;x^*,l^*_f,s^*_f,s^*_u): z \in \bC(w^*)\}.
\eeq
Since Slater's condition holds for $\bC(w^*)$,  the first-order optimality condition of 
\eqref{lim-pt-v} immediately implies that $x^*$ is a KKT point of \eqref{general-dc}.
\end{proof}

\gap

\begin{remark}
For $M=0$, Theorem \ref{v-exact-thm} still holds without the uniform continuity of $F(x)$ in the level set $\cL=\{x\in \cF: F(x) \le F(x^0)\}$.
\end{remark}

\section*{Acknowledgment}

The author would like to thank Ting Kei Pong for bringing his attention to the reference \cite{Rob75}.

\end{document}